\documentclass{amsart}

\usepackage{amssymb, graphicx, hyperref}
\usepackage{stmaryrd}
\usepackage[all]{xy}
\usepackage{xcolor}

\newtheorem{thm}{Theorem}[section]

\newtheorem{cor}[thm]{Corollary}

\newtheorem{prop}[thm]{Proposition}

\theoremstyle{definition}
\newtheorem{defn}[thm]{Definition}

\newtheorem{exmp}[thm]{Example}

\newtheorem{ques}[thm]{Question}    
\newtheorem*{ques*}{Question}

\newtheorem{rem}[thm]{Remark}          

\newtheorem*{ack}{Acknowledgment}      

\newtheorem{defn-thm}[thm]{Definition--Theorem}  
\newtheorem{defn-lem}[thm]{Definition--Lemma}  

\theoremstyle{remark}

\newcommand{\sing}[0]{\operatorname{Sing}}

\newcommand{\pf}[0]{\operatorname{Pf}}

\def\loccoh#1.#2.#3.#4.{H^{#1}_{#2}(#3,#4)}

\DeclareMathAlphabet{\mathchanc}{OT1}{pzc}%
                                {m}{it}

\begin{document}
\bibliographystyle{amsalpha}


\title[Cubic forms having matrix factorizations by Hessian matrices]{Cubic forms having matrix factorizations by Hessian matrices}
\author{Yeongrak Kim}
\email{kim@math.uni-sb.de}
\address{F. Mathematik und Informatik, Universit{\"a}t des Saarlandes, Campus E2.4, D-66123 Saarbr{\"u}cken, Germany}

\begin{abstract}
Using a part of XJC-correspondence by Pirio and Russo, we classify cubic forms $f$ whose Hessian matrices induce matrix factorizations of themselves. When it defines a reduced hypersurface, it satisfies the ``secant--singularity'' correspondence, that is, it coincides with the secant locus of its singular locus. In particular, when $f$ is irreducible, its singular locus is either one of four Severi varieties.
\end{abstract}
\maketitle
\section{Introduction}

Let $k=\mathbb{C}$ be the field of complex numbers, $S=k[x_0, \cdots, x_n]$ be a polynomial ring, and let $f \in S$ be a nonzero homogeneous polynomial. Celebrating Hilbert syzygy theorem implies that the minimal $S$-free resolution of any finitely generated $S$-module terminates in finitely many steps. On the other hand, the minimal $R$-free resolution of a finitely generated module over the hypersurface ring $R$ may not terminate in general. For instance, let $S = k[x]$,  $f=x^2$, and let $R = S/(f)$.  The minimal $R$-free resolution of the module $k = S/(x)$ has infinite length even $k$ is finitely generated:
\[
\cdots \to R(-k) \stackrel{ \cdot x} \longrightarrow R(-k+1) \stackrel {\cdot x} \longrightarrow \cdots \longrightarrow R(-1) \stackrel{\cdot x} \longrightarrow R \to k \to 0.
\]
Eisenbud studied free resolutions over a hypersurface ring $R=S/(f)$, and found that every minimal $R$-free resolution of a finitely generated $R$-module becomes $2$-periodic after finitely many steps \cite[Theorem 6.1]{Eis80}. In particular, there are only two matrices $A$ and $B$ presenting boundary maps of the minimal $R$-free resolution after $(n+1)$-steps. By taking natural representatives, we may assume that the entries of $A$, $B$ are homogeneous polynomials in $S$ so that they induce graded $S$-module homomorphisms between free $S$-modules which degenerate along $f$. This leads to the notion of matrix factorizations of $f$. A pair of matrices $(A,B)$ is called a \emph{matrix factorization of $f$} when $A, B$ induce graded $S$-module homomorphisms and $AB = BA = f \cdot Id$. Notice that if $(A, B)$ and $(A, B^{\prime})$ are two matrix factorizations of $f$, then $B = B^{\prime}$ \cite[Theorem 6.1]{Eis80}, so we will simply say $A$ is a matrix factorization of $f$ when $(A,B)$ is a matrix factorization of $f$ for some matrix $B$. Matrix factorizations are essential to study the finitely generated modules and their resolutions over $R$. Recently, there are several attempts and applications of matrix factorizations including the connections between: ACM and Ulrich sheaves, categories of singularities, and categories of $D$-branes for Landau-Ginzburg $B$-models, as in some pioneering works \cite{ESW03, KL04, Orl04}. 

In most cases, finding a matrix factorization of a homogeneous polynomial $f$ is not obvious. For example, let $f = x_0 ^2 + \cdots + x_n^2$ be a general quadratic polynomial in $S$. Except for trivial matrix factorizations $(1,f)$ and $(f,1)$, a minimal possible nontrivial matrix factorization $A$ is induced by a spinor representation of $f$ (there are $1$ or $2$ up to the parity of $n$), indeed, $A$ is a square matrix of size $2^{\lfloor{\frac{n-2}{2}} \rfloor}$ \cite[Proposition 3.2]{BEH87}. It is known that every homogeneous polynomial $f$ admits a nontrivial matrix factorization \cite{BH87}, however, the smallest size of matrix factorizations is not exactly known even for a generic cubic polynomial. 

In many cases, the construction in \cite{BH87} only ensures a matrix factorization of huge size ($\sim d^{\binom{d+n}{n}}$), which is very far away from a minimal possible matrix factorization. Sometimes, a polynomial can have a matrix factorization with interesting properties. For example, a linearly determinantal polynomial $f$ (\emph{i.e.}, there is a linear matrix $A$ such that $\det A = f$) has a matrix factorization by its determinantal representation $A$. Similarly, a linearly Pfaffian polynomial $f$ (\emph{i.e.}, there is a linear skew-symmetric matrix $A$ such that $\pf A = f$) has a matrix factorization by its Pfaffian representation $A$. Hence, some polynomials admit matrix factorizations of small size, or having ``symmetry''. 

Hence, it is reasonable to ask which polynomials allow ``nice'' matrix factorizations. In this note, we study the Hessian matrix, appearing as a symmetric matrix factorization in the following motivating example. Let $f=x_0 x_1 x_2 \in k[x_0, x_1, x_2]$. It is clear that the Hessian matrix of $f$ gives a linear matrix factorization of $f$ itself. An effective idea to find such cubics is using special Cremona transformations. Indeed, the gradient map (or, the polar map) of $f$ is 
\[
\begin{array}{cccc}
\nabla f : & \mathbb{P}^2 & \dashrightarrow & \mathbb{P}^2 \\
& (x_0:x_1:x_2) & \mapsto & (x_1x_2 : x_0 x_2 : x_0 x_1)
\end{array}
\]
which is a Cremona transformation of $\mathbb{P}^2$. Note that both $\nabla f$ and its inverse are represented by quadratic polynomials. Ein and Shpeherd-Barron classified special quadro-quadratic Cremona transformations \cite[Theorem 2.6]{ESB89}. When the base locus of $\nabla f$ is smooth and irreducible, there are only four possible cases: the base loci are Severi varieties  $v_2 (\mathbb{P}^2) \subset \mathbb{P}^5, \mathbb{P}^2 \times \mathbb{P}^2 \subset \mathbb{P}^8, Gr(2,6) \subset \mathbb{P}^{14}$, and $\mathbb{OP}^2 \subset \mathbb{P}^{26}$, where the maps are given by the system of quadrics through four Severi varieties. They are \emph{prehomogeneous varieties}: there is a Zariski dense $G$-orbit  inside $\mathbb{P}V$ where $G$ is a semisimple algebraic group and $V$ is a representation of $G$. Indeed, such an orbit is the complement of a hypersurface $V(f) \subset \mathbb{P}V$, where $f$ is invariant under the $G$-action. In the case, such an $f$ becomes \emph{homaloidal}, that is, the gradient map $\nabla f$ gives a Cremona transformation $\mathbb{P}V \dashrightarrow \mathbb{P}V$. We refer \cite{ESB89, Dol00, EKP02} for more details and discussion on homaloidal polynomials and prehomogeneous varieties.

The above discussion leads to the following question:

\begin{ques}\label{ques:MatrixFactorizationByHessian}
Which cubic polynomials admit matrix factorizations by their Hessian matrices?
\end{ques}

When it happens, we may wildly expect that it comes from a prehomogeneous variety. Let us compare with four Severi varieties. In each case, there is a prehomogeneous group action $G$, either one of $SL(3), SL(3) \times SL(3), SL(6),$ or $E_6$ and a $G$-invariant cubic polynomial $f$ (unique up to constant multiples). The Severi variety $X$ appears as the singular locus of the cubic hypersurface $V(f) \subset \mathbb{P}^N$, and $V(f)$ coincides with the secant variety of $X$.  We may also ask the ``secant--singularity'' relation occurs for cubic polynomials which verify the above question. 

The main result of this paper is the following classification theorem, which answers to Question \ref{ques:MatrixFactorizationByHessian}:
\begin{thm}[see Theorem \ref{thm:HessianMFisEKPHomaloidal} and Corollary \ref{Cor:ClassificationCubicWithHessianMF}]
Let $f \in k[x_0, \cdots, x_n]$ be a homogeneous cubic form such that $\det \mathcal H (\log f) \neq 0$. Suppose that the Hessian matrix $\mathcal H(f)$ of $f$ induces a matrix factorization of $f$. Then $f$ is linearly equivalent to one of the followings:
\[
\left\{
\begin{array}{cl}
f \ = & x_0^3 \text{ in a single variable;}\\
f \ = & x_0^2x_1 \text{ in two variables;}\\
f \ = & x_0  (x_1^2 + \cdots + x_n^2) \text{ in $(n+1)$ variables;}\\
f \ = & \text{equation of the secant variety of the one of $4$ Severi varieties.} \end{array}
\right.
\]
\end{thm}
\noindent

In particular, when $f$ is defined in $3$ or more variables, the hypersurface $V(f)$ becomes reduced, and it coincides with the secant locus of $\sing V(f)$. Hence, the ``secant--singularity'' relation always holds for such cubic polynomials. Moreover, $f$ is homaloidal and its gradient map gives a quadro-quadric Cremona transformation. The classification follows from an observation that $f$ appears as the norm of some semisimple Jordan algebra of rank $3$ which forms a part of the surprising coincidence called \emph{XJC-correspondence} of Pirio and Russo \cite{PR16}. 

The structure of the paper is as follows. In Section \ref{Sect:Preliminaries}, we first recall homaloidal and EKP-homaloidal polynomials. Note that the gradient map of a cubic homaloidal polynomial gives a Cremona transformation by quadratic polynomials. In several cases, the inverse of such a Cremona transformation is also represented by quadratic polynomials, which we call a quadro-quadric (or $(2,2)$-) Cremona transformation. There is a strong connection between quadro-quadric Cremona transformations and complex Jordan algebras of rank $3$, which consists a part of a beautiful trichotomy called XJC-correspondence. In Section \ref{Sect:MFHessian}, we classify the cubic forms whose Hessian matrices induce matrix factorizations of themselves using the connection between complex Jordan algebras. The key idea is to construct a semisimple Jordan algebra of rank $3$ from the Hessian matrix whose norm coincides with the given cubic polynomial. We finish by a few examples and further questions for higher degrees.

\section{Preliminaries}\label{Sect:Preliminaries}
We recall some helpful notions and facts. Recall that each Severi variety is associated to a certain prehomogeneous group action, which has a unique invariant cubic $f$ (up to constant multiples). As discussed above, the gradient map $\nabla f : \mathbb{P}^N \dashrightarrow \mathbb{P}^N$ becomes a quadro-quadric Cremona transformation, that is, $\nabla f$ is birational, and both $\nabla f$ and its inverse map are represented by quadratic polynomials. It is interesting to study polynomials whose gradient map is birational.

\begin{defn}
Let $f \in k[x_0 \cdots, x_n]$ be a homogeneous polynomial of degree $d$. $f$ is called \emph{homaloidal} if its partial derivatives $\left( \frac{\partial{f}}{\partial x_0}, \cdots, \frac{\partial{f}}{\partial x_n} \right)$ define a Cremona transformation, \emph{i.e.}, give a birational map $\nabla f : \mathbb{P}^n \dashrightarrow \mathbb{P}^n$. 

Assume furthermore that the Hessian determinant of $(\log f)$ is nonzero. $f$ is called \emph{EKP-homaloidal} if its multiplicative Legendre transformation $g$ is a polynomial \cite{EKP02}. Note that the \emph{multiplicative Legendre transformation $g$ of $f$} is a homogeneous function which satisfies 
\[
g \left( \nabla \log f \right) = g \left( f^{-1} \frac{\partial{f}}{\partial x_0}, \cdots, f^{-1} \frac{\partial{f}}{\partial x_n} \right) = \frac{1}{f(x_0, \cdots, x_n)}.
\]
\end{defn}
Since the Hessian determinant of $\log f$ is nonzero, the multiplicative Legendre transformation is well-defined on an analytic neighborhood at each general point of $\mathbb{P}^n$ thanks to the inverse function theorem. The multiplicative Legendre transformation needs not to be algebraic in general, but it is always a homogeneous function of the degree same as $f$.

Let $g$ be the multiplicative Legendre transform of a homogeneous polynomial $f$ of degree $d$, and let $y_i := \frac{\partial f}{\partial x_i}$ be the $i$-th partial derivative of $f$. Then, we have $g(y_0, \cdots, y_n) = f(x_0, \cdots, x_n)^{d-1}$ by definition, and hence
\begin{eqnarray*}
\frac{\partial g}{\partial y_i} & = & (d-1) f(x_0, \cdots, x_n)^{d-2} \sum_{j} \frac {\partial f}{\partial x_j} \frac{\partial x_j}{\partial y_i} \\
& = & (d-1) f(x_0, \cdots, x_n)^{d-2} \sum_j y_j \frac{\partial x_i}{\partial y_j} \\
& = & f(x_0, \cdots, x_n)^{d-2} x_i
\end{eqnarray*}
thanks to Euler's formula. Thus, $\nabla f$ is a Cremona transformation of type $(d-1, d-1)$ with its inverse $\nabla g$. In particular, EKP-homaloidal implies homaloidal, and EKP-homaloidal cubics always give quadro-quadric Cremona transformations.

In practice, it is not quite easy to determine whether a given polynomial $f$ gives a quadro-quadric Cremona transformation as its gradient map. Instead of computing the inverse Cremona transformation and checking whether it is quadratic, we take a small detour, called \emph{XJC-correspondence} \cite{PR16}. It says that the following objects are in $1-1$ correspondences:
\begin{enumerate}
\item irreducible $3$-RC varieties $X$ covered by twisted cubics; 
\item complex Jordan algebras of rank $3$; 
\item quadro-quadric Cremona transformations. 
\end{enumerate}
In this note, we are particularly interested in the equivalence between complex Jordan algebras of rank $3$ and quadro-quadric Cremona transformations. Indeed, the main theorem of this note follows from the classification of semisimple complex Jordan algebras (of rank $3$). 

Let us briefly recall about Jordan algebras. Note that a complex Jordan algebra is a commutative complex algebra $J=(V, \ast, e)$ with the unity $e$ satisfying the Jordan identity
\[
u^2 \ast (u \ast v) = u \ast (u^2 \ast v)
\]
for every $u, v \in J$. In general, a Jordan algebra is not associative but power associative (\emph{i.e.}, $u^{m+n}=u^m \ast u^n$ for any $m,n$).

Note that a power associative algebra admits a number of notions and properties corresponding to the algebra of square matrices.  Let $\mathbb{C}$-algebra $(V, \ast, e)$ be a power associative algebra, and let $u \in V$ be an element. There is a minimal polynomial $m_u(t)$ which is the monic generator of the kernel of the evaluation map
\begin{eqnarray*}
\phi_u : \mathbb{C}[t] & \to & \mathbb{C}[u]. \\
t & \mapsto & u
\end{eqnarray*}
We define the rank of $(V, \ast, e)$ as $\max \{ \deg m_u(t), u \in V\}$. In particular, there is an analogous statement of Cayley-Hamilton theorem for a complex Jordan algebra $J = (V, \ast, e)$ of rank $3$. Indeed, there is a linear form $T \in J^{*}$ (generic trace), a quadratic form $S \in Sym^2(J^{*})$, and a cubic form $N \in Sym^3(J^{*})$ so that we have a universal minimal polynomial
\[
u^3 - T(u) u^2 + S(u) u - N(u) e = 0
\]
for every $u \in J$. $N$ is called the \emph{norm of $J$}.

We define the \emph{adjoint of $u$} as $u^{\#} := u^2 - T(u)u + S(u)e$. The adjoint and the norm have similar roles as the adjoint and the determinant for usual matrices, in particular, the Laplace formula
\[
u \ast u^{\#} = u^{\#} \ast u = N(u) e
\]
holds for every $u \in J$. We leave \cite{McC78} and \cite[Chapter 6]{Rus16} for more details on power associative and Jordan algebras. 

Let us briefly describe how Jordan algebras and Cremona transformations are related.  If we have a complex Jordan algebra $J= (V, \ast, e)$ of rank $3$, the adjoint map $u \mapsto u^{\#}$ is a quadratic map $V \to V$. Thanks to the Laplace formula, it gives a Cremona involution $\mathbb{P}V \dashrightarrow \mathbb{P}V$, well-defined for elements $u$ such that $N(u) \neq 0$. Note that in the case of Severi varieties, after a suitable linear change of coordinates if necessary, the gradient map of the invariant cubic $f$ becomes a Cremona involution \cite[Theorem 2.8]{ESB89}. In particular, there is a complex Jordan algebra structure with norm $f$, and its gradient map $\nabla f$ plays the role of the adjoint.

We refer \cite[Section 2.2, Theorem 3.4]{PR16} for the detailed description of the converse correspondence. For more interested readers, we also refer Mukai's note \cite{Muk98} which explains a connection between semisimple Jordan algebras of rank $3$ and Legendre varieties.

\section{Cubic form whose Hessian gives its matrix factorization}\label{Sect:MFHessian}
In this note, we study a cubic form which admits a matrix factorization by its Hessian matrix. Such a cubic form is very uncommon even the number of variables is small. Note that the property does not change by a PGL-action, we deal such cubic forms up to linear equivalence. We begin with a few simple examples.

\begin{exmp}\label{exmp:HessianGivesMF}\ 
\begin{enumerate}
\item Let $f = x_0 x_1 x_2 \in k[x_0,x_1,x_2]$. The Hessian $\mathcal H(f)$ of $f$ is
\[
\left(
\begin{array}{ccc}
0 & x_2 & x_1 \\
x_2 & 0 & x_0 \\
x_1 & x_0 & 0 
\end{array}
\right).
\]
It gives a matrix factorization of $f$ since the matrix
\[
\mathcal Q(f) := \frac{1}{2} 
\left(
\begin{array}{ccc}
-{x_0^2} & x_0 x_1 & x_0 x_2 \\
x_0 x_1 & -x_1^2 & x_1 x_2 \\
x_0 x_2 & x_1 x_2 & -x_2^2 
\end{array}
\right)
\]
satisfies $\mathcal H(f) \cdot \mathcal Q(f) = \mathcal Q(f) \cdot \mathcal H(f) = f \cdot Id$. 

\item Let $f = x_0^3+x_1^3+x_2^3-3 \lambda x_0 x_1 x_2$, $\lambda \in \mathbb{C}$ be a Hesse cubic form. One can check that $f$ admits a linear matrix factorization by its Hessian if only if $\lambda^3=1$. In particular, it factors completely into a product of $3$ distinct linear forms, and hence it is linearly equivalent to the first example.

\item Let $Z=v_2 (\mathbb{P}^2)$ be the Veronese surface in $\mathbb{P}^5$, and let $X=Sec(Z)$ be its secant variety. It is well-known that the ideal of $Z$ and $X$ are generated by the 2-minors and the determinant of the symmetric matrix
\[
A = 
\left(
\begin{array}{ccc}
x_0 & x_1 & x_2 \\
x_1 & x_3 & x_4 \\
x_2 & x_4 & x_5
\end{array}
\right)
\]
respectively, in particular, $X$ is a cubic hypersurface defined by $f = \det A$. One can check that the Hessian matrix $\mathcal H(f)$ of $f$ gives a matrix factorization of itself.
\end{enumerate}
\end{exmp}

Inspired by the case of Severi varieties as in \cite{ESB89}, it is natural to consider the relations between Cremona transformations. Let $X = V(f)$ be the cubic hypersurface which is the secant variety of one of $4$ Severi varieties. As discussed above, the partial derivatives of $f$ induce a Cremona transformation $\tau = \left( \frac{\partial{f}}{\partial x_0}, \cdots, \frac{\partial{f}}{\partial x_n} \right)$, and furthermore, $\tau^2=id$, \emph{i.e.}, $\tau$ is an involution (after a linear change of coordinates if necessary) \cite[Theorem 2.8]{ESB89}. In particular, the base locus $Z (=  \sing X)$ of $\tau$  recovers the Severi variety, and the gradient map $\tau$ gives a quadro-quadric Cremona transformation. In fact, $f$ is EKP-homaloidal, which can be found as Example $3 - 6$ from the list \cite[Examples in p.38,  Theorem 3.10]{EKP02}. In the case, one can check that the Hessian matrix $\mathcal H(f)$ gives a matrix factorization of $f$ itself (see the arguments in \cite[Remark 3.5]{KS19}). Thus, our question turns out to be:

\begin{ques}\label{ques:HessianMFandEKPHomaloidal}
Let $f$ be a cubic polynomial whose Hessian gives a matrix factorization of itself. Is $f$ EKP-homaloidal?
\end{ques}

In the above examples, every $V(f)$ coincides with the secant variety of its singular locus. Hence, we may also ask the following geometric question:

\begin{ques}\label{ques:HessianAndSecant}
Let $X=V(f)$ be a cubic hypersurface and $Z=\sing(X)$ be its singular locus. Suppose that the Hessian matrix $\mathcal H(f)$ forms a matrix factorization of $f$. Does the secant variety $Sec(Z)$ coincide with $X$?
\end{ques}

The converse of the question is negative. Let $Z$ be the rational normal curve in $\mathbb{P}^4$, and let $X$ be its secant variety. Note that $X$ is a cubic hypersurface and $\sing(X)=Z$. In this case, one can check that the Hessian matrix of $X$ does not give a matrix factorization since the cokernel module of the matrix is not supported on $X$.

Without any information about its matrix factorization, we have the following elementary proposition which connects the Hessian and the secant variety of the singular locus.

\begin{prop}
Let $X = V(f) \subset \mathbb{P}^{n}$ be a cubic hypersurface, and let $Z=\sing(X) \neq \emptyset$ be its singular locus. Then the secant variety $Sec(Z)$ is contained in $X$, and the Hessian matrix of $X$ is not of full rank along $Sec(Z)$ .
\end{prop}
\begin{proof}
Suppose first that $Z$ is not a single point. Let $P, Q \in Z$ be two distinct points, and let $\ell = \left< P, Q \right>$ be the line passing through $P$ and $Q$. Note that $\ell$ is contained in $X$ since the length of $\ell \cap X$ is at least 4. In particular, $Sec(Z) \subseteq X$.

After a certain coordinate change, we may assume that $P=[1:0:\cdots:0]$ and $Q=[0:1:0:\cdots : 0]$. Since $P$ and $Q$ are singular points of $X$, we may write $f$ as follows:
\[
f=\sum_{i=2}^{n}c_i x_0 x_1 x_i + x_0 g_2 (x_2, \cdots, x_{n}) + x_1 g^{\prime}_2 (x_2, \cdots, x_{n}) + h_3(x_2, \cdots, x_{n})
\]
where $g_2, g^{\prime}_2, h_3$ are polynomials in $x_2, \cdots, x_{n}$ of degree $2, 2, 3$, respectively. At a point $[a:b:0:\cdots:0]$ on $\ell$, the first two rows of Hessian of $F$ are
\[
\left(
\begin{array}{cccccc}
0 & 0 & c_2 b & c_3 b & \cdots & c_{n}b \\
0 & 0 & c_2 a & c_3 a & \cdots & c_{n}a \\
\vdots & \vdots & \vdots & \vdots & \cdots & \vdots
\end{array}
\right)
\]
hence the first two rows are linearly dependent. 

When $Z$ is a single point, then the secant variety of $Z$ is $Z$ itself. We may assume that $P = [1:0:\cdots : 0]$ and $Z = \{P\}$, hence,
\[
f = x_0 g_2(x_1, \cdots, x_n ) + h_3 (x_1, \cdots, x_n)
\]
for some polynomials $g_2, h_3$ in $x_1, \cdots, x_n$ of degree $2, 3$, respectively. It is clear that the first row of the Hessian at $P$ is 0.
\end{proof}

\begin{cor}
Let $X=V(f)$ be a reduced cubic hypersurface and $Z=\sing(X)$ be its singular locus. Suppose that $Sec(Z)=X$. Then the determinant of the Hessian is divisible by $f$.
\end{cor}

Note that if an $(m \times m)$ matrix $A$ gives a matrix factorization of $f$, then $\det A$ divides $f^m$. Hence, when $f$ is irreducible, the only possible values for $\det A$ are powers of $f$. Notice that the determinant of a matrix factorization $A$ of $f$ cannot be zero, whereas the determinant of the Hessian matrix can vanish sometimes. We refer \cite{CRS08, GR15} for discussion on polynomials with vanishing Hessian determinants. 

Let us go back to our main question. Under a mild nondegenerate condition on $\log f$, we answer to Question \ref{ques:HessianMFandEKPHomaloidal}, by a short detour to cubic Jordan algebras.

\begin{thm}\label{thm:HessianMFisEKPHomaloidal}
Let $f$ be a homogeneous cubic form such that the Hessian determinant $\det \mathcal H (\log f)$ is not identically $0$. Then the Hessian of $f$ induces a matrix factorization of $f$ if and only if $f$ is EKP-homaloidal.
\end{thm}
\begin{proof}
$(\Mapsto)$ We follow the arguments in the proof of \cite[Theorem 3.4]{PR16}. Let $y_i = \frac {\partial f}{\partial x_i}$ be the $i$-th partial derivative of $f$. Let $V$ be the $(n+1)$-dimensional vector space so that $\mathbb{P}V$ is the ambient projective space with coordinates $x_0, \cdots, x_n$. Note that the map $\iota := d (\log f) = (f^{-1} y_0 , \cdots, f^{-1} y_n)$ is a rational map which is homogeneous of degree $-1$. To construct a cubic Jordan algebra with the norm $f$, the key is to build a quadratic affine morphism $P : V \to End (V)$ given by $P (u_0, \cdots, u_n) = - (d \iota)^{-1}_{(u_0, \cdots, u_n)}$. If then, it is straightforward that
\[
P(u,v)=P(u+v)-P(u)-P(v)
\]
is bilinear in $u, v \in V$, and hence a result of McCrimmon \cite[Theorem 4.4]{McC77} implies that
\[
u \ast v := \frac{1}{2}P(u,v)(e)
\]
satisfies the Jordan identity, where $e \in V$ is the identity element. Note that the existence of the unity $e$ is ensured by arguments in the first paragraph in \cite[Proof of Theorem 3.4]{PR16}. Indeed, $(V,\ast, e)$ is a cubic Jordan algebra equipped with the norm $f$, and the adjoint map is given by the gradient map $u \mapsto (y_0 (u), \cdots, y_n (u))$. 

By the assumption, the Hessian $\mathcal H(f)$ gives a matrix factorization $(\mathcal H(f), \mathcal Q(f))$ of $f$, where the entries of $\mathcal Q(f) = f  \cdot \mathcal H(f)^{-1}$ are quadratic polynomials in $x_0, \cdots, x_n$. It is enough to show that the quadratic rational map $- (d \iota)^{-1}$ is indeed an affine morphism, \emph{i.e.,}, the entries of $-(d \iota)^{-1}$ are given by polynomials. Since 
\begin{eqnarray*}
(d \iota)_{(i,j)} = \frac{\partial (\log f)}{\partial x_i \partial x_j} & = & \frac{\partial (f^{-1} y_j)}{\partial x_i} \\
& = & \frac{1}{f^2} \left( f \frac{\partial y_j}{\partial x_i} - y_i y_j \right)
\end{eqnarray*}
equals to the $(i,j)$-th entry of the matrix $[f^{-1} \mathcal H(f)  - f^{-2} (y_0, \cdots, y_n)^T (y_0, \cdots, y_n)]$, we immediately check that
$[\mathcal Q(f) - \frac{1}{2} (x_0, \cdots, x_n)^T (x_0, \cdots, x_n)]$ is the inverse of $(d \iota)$ thanks to the Euler formula
\[
\begin{array}{rcl}
\mathcal H(f) (x_0, \cdots, x_n)^T & = & 2 \cdot (y_0, \cdots, y_n)^T \\
(y_0, \cdots, y_n) (x_0, \cdots, x_n)^T & = & 3 \cdot f \cdot (x_0, \cdots, x_n). \\
\end{array}
\]
In particular, 
\[
P=-(d \iota)^{-1} = \left[ \frac{1}{2}(x_0, \cdots, x_n)^T (x_0, \cdots, x_n) - \mathcal Q(f) \right]
\]
 is given by a square matrix whose entries are homogeneous quadratic polynomials in $x_0, \cdots, x_n$. 

Note that the hypersurface $V(f)\subset \mathbb{P}^n$ cannot be a cone since its Hessian $\mathcal H(f)$ has nonvanishing determinant. Hence, the radical $V(d^2 f)  \subset \mathbb{P}^n$, which is the vertex of the cone $V(f)$ \cite[Proposition 4.4]{PR16}, must be empty. This implies that the cubic Jordan algebra $(V,\ast, e)$ we constructed is semisimple, and thus the conclusion follows from \cite[Corollary 4.6]{PR16} as desired.

$(\Mapsfrom)$ Let $y_i = \frac {\partial f}{\partial x_i}$ be the $i$-th partial derivative of $f$, which is a homogeneous quadratic polynomial in $x_i$'s. Recall that the multiplicative Legendre transformation of $f$ is a homogeneous function $g$ such that
\[
	g\left(  \frac{y_0}{f(x_0, \cdots, x_n)}, \cdots, \frac{y_n}{f(x_0, \cdots, x_n)} \right) = \frac{1}{f(x_0, \cdots, x_n)}.
\]
Since the degree of $g$ equals to the degree of $f$, we have the following identity
\[
g(y_0, \cdots, y_n) = f(x_0, \cdots, x_n)^2.
\]
Thanks to the Euler formula, it is easy to see that the 2nd derivatives of $g$ are given by:
\[
\frac{\partial^2 g}{\partial y_i \partial y_j} = \frac{\partial (x_j f(x_0, \cdots, x_n))}{\partial y_i} = f \cdot \left( \frac{\partial x_j}{\partial y_i} \right) + \frac{1}{2}x_i x_j.
\]
Since $f$ is EKP-homaloidal, the multiplicative Legendre transform $g$ becomes a cubic polynomial, and hence the above 2nd derivative of $g$ is a linear polynomial in $y_i$'s (= quadratic polynomial in $x_i$'s). Thus, we conclude that each entry of the matrix  
\[
\mathcal Q(f)_{{(i,j)}} := f \cdot \left( \frac{\partial x_j}{\partial y_i} \right) = \frac{\partial^2 g}{\partial y_i \partial y_j} - \frac{1}{2}x_i x_j
\]
 is a homogeneous quadratic polynomial in $x_i$'s. Since the matrix $\left( \frac{\partial x_j}{\partial y_i} \right)_{(i,j)}$ is the inverse of the Hessian $\mathcal H(f) = \left( \frac{\partial y_i}{\partial x_j} \right)_{(i,j)}$, we have $\mathcal Q(f) \mathcal H(f) = \mathcal H(f) \mathcal Q(f)= f \cdot Id$, \emph{i.e.}, the Hessian matrix $\mathcal H(f)$ gives a matrix factorization of $f$.
\end{proof}

\begin{rem}
There are a few degenerate quadro-quadric Cremona transformations when the number of variables is small. For example, let $f=x_0^2 x_1$ be a cubic form in $2$ variables. It is clear that the Hessian of $f$ induces a matrix factorization of $f$, however, the gradient map $(2x_0 x_1, x_0^2)$ is composed of quadratic polynomials which have a nontrivial common divisor. Hence, in a strict sense, it is not a $(2,2)$-Cremona transformation, but an $(1,1)$-Cremona transformation after dividing the common factor $x_0$ (which is called a fake quadro-quadric Cremona transformation, see \cite[Example 2.2-(2)]{PR16}).
\end{rem}

The above correspondence immediately gives the complete classification of cubic forms whose Hessian matrices induce matrix factorizations of themselves (see the list \cite[Theorem 3.10]{EKP02}, \cite [Table 1]{PR16}). In particular, we also give an affirmative answer to Question \ref{ques:HessianAndSecant} by plugging in each case.

\begin{cor}\label{Cor:ClassificationCubicWithHessianMF}
Let $f$ be a homogeneous cubic form such that $\det \mathcal H (\log f) \neq 0$. Suppose that the Hessian of $f$ induces a matrix factorization of $f$. Then $f$ is linearly equivalent to one of the followings:
\[
\left\{
\begin{array}{cl}
f \ = & x_0^3 \text{ in a single variable;}\\
f \ = & x_0^2x_1 \text{ in two variables;}\\
f \ = & x_0  (x_1^2 + \cdots + x_n^2) \text{ in $(n+1)$ variables;}\\
f \ = & \text{equation of the secant variety of the one of $4$ Severi varieties.} \end{array}
\right.
\]
In particular, when $f$ is defined in $3$ or more variables, then $X=V(f)$ coincides with the secant variety of its singular locus $Z=\sing(X)$. 
\end{cor}
\begin{proof}
By Theorem \ref{thm:HessianMFisEKPHomaloidal}, such $f$ is EKP-homaloidal.  When $f$ is written in $3$ or more variables, the result immediately follows from \cite[Corollary 4.6]{PR16} and the corresponding classification \cite[Table 1]{PR16}. In the case of a single variable, there is only one homogeneous cubic form $f=x_0^3$ up to linear equivalence. Clearly, its 2nd derivative $6x_0$ gives a matrix factorization (by $1 \times 1$ matrices). Finally, in the case of two variables, $f$ completely decomposes into a multiple of linear forms again. One can also check that the only possible case is that $f$ is a multiple of the square of a linear form and another linear form, namely, $f = x_0^2 x_1$. 
\end{proof}

We finish this note with a short remark on higher degrees. In fact, \cite[Theorem 2.8]{ESB89} implies much more than cubics, when we play with a regular prehomogeneous variety and the corresponding invariant hypersurface $V(f)$. Suppose we have an irreducible regular prehomogeneous representation $V \simeq \mathbb{C}^{n+1}$ of a semisimple group $G$, as in \cite[Theorem 2.8]{ESB89}. Take $f$ be a $G$-invariant polynomial, and let $d$ be its degree. Following the arguments of Ein and Shepherd-Barron, the (signed) partial derivatives defines a Cremona involution. Since we assumed that the Hessian determinant $\det \mathcal H(f)$ is nonzero, it should divide some power of $f$. In particular, when $f$ is irreducible, then the Hessian $\mathcal H(f)$ gives a matrix factorization of $f^r$ for some $r>0$. When $d = 3$, the Hessian $\mathcal H(f)$ gives a matrix factorization of $f$ itself as we discussed. On the other hand, when $d>3$, the Hessian $\mathcal H(f)$ needs not to give a matrix factorization of $f$; we have to carefully observe the exponent $r$. We address a few higher degree examples whose Hessian gives a matrix factorization for some power of $f$.

\begin{exmp}\ 
\begin{enumerate}
\item Let $S=k[x_{00}, \cdots, x_{22}, z]$ and $f$ be the multiple of $z$ and the determinant of the $3 \times 3$ matrix consisted of $x_{00} \cdots, x_{22}$, so that $f$ is invariant under a $SL(3) \times \mathbb{C}^{\times}$-action. One can check that $f$ is a quartic homogeneous polynomial whose Hessian gives a matrix factorization of $f$ itself.

\item Let $S=k[x_{000}, \cdots, x_{111}]$ and $f$ be the hyperdeterminant of the $2 \times 2 \times 2$ hypermatrix with entries given by the coordinates. Then $f$ is a quartic homogeneous polynomial which is invariant under the $SL(2) \times SL(2) \times SL(2)$-action. Following the argument of Ein and Shepherd-Barron, the Hessian $\mathcal H(f)$ gives a matrix factorization of $f^2$. One can check that the Hessian does not induce a matrix factorization of $f$ since its cokernel is annihilated by $f^2$, but not by $f$. 

\item Let $S=k[x_{000}, \cdots, x_{211}]$ and $f$ be the hyperdeterminant of the $3 \times 2 \times 2$ hypermatrix with entries given by the coordinates. As similar as above, $f$ is a sextic homogeneous polynomial which is invariant under $SL(3) \times SL(2) \times SL(2)$-action. Following the argument of Ein and Shepherd-Barron, one can check that the Hessian $\mathcal H(f)$ gives a matrix factorization of $f^4$. The exponent $r=4$ is not the minimal one; the smallest power of $f$ such that the Hessian gives a matrix factorization is $2$ since the cokernel module of $\mathcal H(f)$ is annihilated by $f^2$.
\end{enumerate}
\end{exmp}

\begin{ques}
Suppose that $f$ is a homogeneous polynomial of degree $d$ whose Hessian $\mathcal H(f)$ forms a matrix factorization of some power of $f$. Let $r>0$ be the smallest power. What is the algebro-geometric meaning of $r$? Is $f$ an invariant polynomial of a prehomogeneous action?
\end{ques}

\begin{ack}
The author thanks Frank-Olaf Schreyer and Francesco Russo for the helpful discussion. This work was supported by Project I.6 of SFB-TRR 195 ``Symbolic Tools in Mathematics and their Application'' of the German Research Foundation (DFG).
\end{ack}

\def\cprime{$'$} \def\cprime{$'$} \def\cprime{$'$} \def\cprime{$'$}
  \def\cprime{$'$} \def\cprime{$'$} \def\dbar{\leavevmode\hbox to
  0pt{\hskip.2ex \accent"16\hss}d} \def\cprime{$'$} \def\cprime{$'$}
  \def\polhk#1{\setbox0=\hbox{#1}{\ooalign{\hidewidth
  \lower1.5ex\hbox{`}\hidewidth\crcr\unhbox0}}} \def\cprime{$'$}
  \def\cprime{$'$} \def\cprime{$'$} \def\cprime{$'$}
  \def\polhk#1{\setbox0=\hbox{#1}{\ooalign{\hidewidth
  \lower1.5ex\hbox{`}\hidewidth\crcr\unhbox0}}} \def\cdprime{$''$}
  \def\cprime{$'$} \def\cprime{$'$} \def\cprime{$'$} \def\cprime{$'$}
\providecommand{\bysame}{\leavevmode\hbox to3em{\hrulefill}\thinspace}
\providecommand{\MR}{\relax\ifhmode\unskip\space\fi MR }
\providecommand{\MRhref}[2]{%
  \href{http://www.ams.org/mathscinet-getitem?mr=#1}{#2}
}
\providecommand{\href}[2]{#2}

\vskip1cm

\end{document}